\documentclass{article}[12pt]
\usepackage{amsmath}
\usepackage{amsfonts}
\usepackage{amssymb}
\usepackage{appendix}
\usepackage{amsthm}
\usepackage{array}
\usepackage{footnote}
\usepackage{hyperref}
\newcolumntype{L}{>$l<$}
\usepackage{fullpage}
\usepackage{graphicx}
\usepackage{epstopdf}
\usepackage[numbers,colon]{natbib}
\usepackage{enumerate}
\newtheorem{Thm}{Theorem}[section]

\newtheorem{Lem}[Thm]{Lemma}

\newtheorem{Pro}[Thm]{Proposition}

\makeatletter
\def\blfootnote{\xdef\@thefnmark{}\@footnotetext}
\makeatother

\theoremstyle{definition}
\newtheorem{Def}[Thm]{Definition}

\theoremstyle{remark}
\newtheorem{Rem}[Thm]{\bf{Remark}}

\newcommand{\ConvFDD}{\overset{f.d.d.}{\longrightarrow}}

\newcommand{\Cov}{\mathrm{Cov}}

\newcommand{\E}{\mathbb{E}}

\title{Sensitivity of the Hermite rank}

\author{
Shuyang Bai \\
University of Georgia
\and
 Murad S. Taqqu\\
Boston University
}

\begin{document}
\maketitle
\begin{abstract}
\noindent The Hermite rank appears in   limit theorems involving long memory. We show that an Hermite rank  higher than one is unstable when the data is slightly perturbed by transformations such as shift and scaling. We carry out a ``near higher order rank analysis'' to illustrate how the limit theorems are affected by a  shift perturbation that is decreasing in size.    We also consider the case where the deterministic shift is replaced by centering with respect to the sample mean.  The paper is a   companion of \citet{bai:taqqu:2017:instability} which discusses the instability of the Hermite rank in the statistical context.
\end{abstract}
\blfootnote{
\begin{flushleft}
\textbf{Key words:} Long-range dependence; Long memory; Hermite rank; power rank;  limit theorem; instability;
\end{flushleft}
\textbf{2010 AMS Classification:} 62M10, 	60F05   \\
}

\section{Introduction}

A stationary sequence $\{X(n)\}$ with finite variance is said to have \emph{long memory} or  \emph{long-range dependence}, if
\begin{equation}\label{eq:cov LRD}
\Cov[X(n),X(0)]\approx n^{2H-2}
\end{equation}
as $n\rightarrow\infty$, where $\approx$ means asymptotic equivalence up to a positive constant\footnote{In general, we can allow inserting a slowly varying factor (\citet{bingham:goldie:teugels:1989:regular}), e.g.,  a logarithmic function,   in the asymptotic relation (\ref{eq:cov LRD}), but we shall not do that for simplicity. }, the parameter  $H\in (1/2,1)$ is called the \emph{Hurst index}.  See, e.g., the recent monographs   \citet{giraitis:koul:surgailis:2009:large}, \citet{beran:2013:long}, \citet{samorodnitsky:2016:stochastic} and \citet{pipiras:taqqu:2017:long}  for more information on the notion long memory.

The study of the asymptotic behavior of partial sums of long-memory sequences has been of great interest in probability and statistics. In particular, complete results have been obtained for the following  class of models called \emph{Gaussian subordination}. Let $\{Y(n)\}$ be a stationary Gaussian process with long memory in the sense of (\ref{eq:cov LRD}), which we suppose to have mean $0$ and variance $1$ without loss of generality. Now consider the transformed process
\begin{equation}
X(n)=G(Y(n)),
\end{equation}
where $G(\cdot):\mathbb{R}\rightarrow\mathbb{R}$ is a function such that $\E X(n)^2=\E G(Y(n))^2<\infty$.  The goal is to establish limit theorems for the normalized sums  of $X(n)$, the transformed stationary process.

To develop the limit theorems, one has to use the notion of Hermite rank, which is
  an integer  attached to the function $G(\cdot)$. It is defined as follows.
Let $Z$ denote a standard Gaussian random variable  and let $\gamma(dx)=\phi(x)dx:=(2\pi)^{-1/2} e^{-x^2/2} dx$  denote its distribution.  The Hermite rank is associated with the orthogonal decomposition  of $G(\cdot)$ in the space
\[L^2(\gamma)=L^2(\phi)=\{G(\cdot):~\E G(Z)^2<\infty\}=\left\{G(\cdot):~\int_{\mathbb{R}} G(z)^2 \phi(z)dz<\infty\right\}
\]  into \emph{Hermite polynomials}, which are defined as $H_0(x)=1$, $H_1(x)=x$, $H_2(x)=x^2-1$, and more generally,
\begin{equation}\label{eq:Herm Poly}
H_m(x)=(-1)^m e^{x^2/2}\frac{d}{dx^m}e^{-x^2/2} \quad\text{ for }\quad m\ge 1.
\end{equation}
 Then $\{\frac{1}{\sqrt{m!}}H_m(\cdot), m\ge 0\}$ forms an orthonormal basis of $L^2(\gamma)$ (see \citet{pipiras:taqqu:2017:long}, Proposition 5.1.3).

\begin{Def}\label{Def:Herm rank}
  The  \emph{Hermite  rank}  $k$ of $G(\cdot)\in L^2(\gamma)$ is defined as
\begin{equation}\label{eq:def herm rank}
k=\inf\left\{m\ge 1: ~\E G(Z) H_m(Z)=\int_{\mathbb{R}}G(x)H_m(x)\gamma(dx)\neq 0 \right\},
\end{equation}
where $H_m(\cdot)$ is the $m$-th order Hermite polynomial.
Equivalently, $k$ is the order of the first nonzero coefficient in the $L^2(\gamma)$-expansion:
\begin{equation}\label{eq:def herm rank expan}
G(\cdot)-\E G(Z)=\sum_{m=k}^\infty c_m H_m(\cdot), \quad k\ge 1,
\end{equation}
that is, $c_m=0$ for $m<k$ and $c_k\neq 0$ where
\begin{equation}\label{eq:Herm Coeff}
c_m={\E[G(Z)H_m(Z)]\over m!} ~~\text{ for }~~m\ge 0.
\end{equation}
\end{Def}
The following celebrated results due to \citet{dobrushin:major:1979:non},  \citet{taqqu:1979:convergence} and \citet{breuer:major:1983:central} have been established involving the Hermite rank.
\begin{Thm}\label{Thm:nclt gaussian}
Suppose that $G(\cdot)\in L^2(\gamma)$ has Hermite rank $k\ge 1$.  Then the following conclusions hold.

\noindent $\bullet$   Central limit case: suppose that $1/2<H<1-\frac{1}{2k}$ (this implies $k\ge 2$).  Then $\{G(X(n))\}$ has short memory in the sense that   \[\sigma^2:=\sum_{n=-\infty}^\infty \Cov[G(X(n)),G(X(0))]\] converges absolutely and
\begin{equation}\label{eq:Herm limit}
\frac{1}{N^{1/2}}\sum_{n=1}^{[Nt]} \Big( G(X(n))-\E G(X(n)) \Big) \ConvFDD \sigma B(t),\quad t\ge 0,
\end{equation}
where $\ConvFDD$ denotes convergence of the finite-dimensional distributions and $B(t)$ is the standard Brownian motion.

\noindent $\bullet$ Non-central limit case:  suppose that $1-\frac{1}{2k}<H<1$.
$\{G(X(n))\}$ has long memory with Hurst index:
\begin{equation}\label{eq:H_G general}
H_G=(H-1)k+1\in \left(\frac{1}{2},1\right).
\end{equation}
Furthermore, as $N\rightarrow\infty$, we have\footnote{In fact,  we have weak convergence in the space $D[0,1]$  with uniform metric.}
\begin{equation}\label{eq:Herm limit2}
\frac{1}{N^{H_G}}\sum_{n=1}^{[Nt]} \Big( G(X(n))-\E G(X(n)) \Big) \ConvFDD c Z_{H,k}(t),
\end{equation}
for some $c\neq 0$,  and
\begin{equation}\label{eq:Herm process}
Z_{H,k}(t)= \int_{\mathbb{R}^k}' \left[ \int_0^t \prod_{j=1}^k (s-x_j)_+^{H-3/2} ds\right]  B(dx_1)\ldots B(dx_k),
\end{equation}
is the so-called $k$-th order Hermite process,
where
 $\int_{\mathbb{R}^k}'   \cdot  \ B(dx_1)\ldots B(dx_k)$ denotes the $k$-tuple Wiener-It\^o integral with respect to the standard Brownian motion $B(\cdot)$ (\citet{major:2014:multiple}).
\end{Thm}

 \citet{bai:taqqu:2017:instability} point out  that despite  the probabilistic interest of Theorem \ref{Thm:nclt gaussian}, its  straightforward application  to large-sample statistical theories can be problematic. This is due to a   strong instability feature in the notion Hermite rank. This paper can be viewed as a technical companion to \citet{bai:taqqu:2017:instability} containing some mathematical characterization of such instability and related problems. In particular, the paper contains the following results:
\begin{enumerate}
\item The instability of the Hermite rank with respect to transformations including shift, scaling, etc.

\item A ``near higher order rank'' analysis of Theorem \ref{Thm:nclt gaussian} perturbed by a diminishing shift.

\item Coincidence of Hermite rank and power rank (\citet{ho:hsing:1997:limit}) in the Gaussian case.
\end{enumerate}
The third result is not closely related to the first two, but it is obtained as a direct byproduct of the analysis of the instability problems.
The  paper is organized as follows:
the   results described above are stated in Sections  \ref{Sec:affine}, \ref{Sec:near} and \ref{Sec:coin}.    Section \ref{Sec:aux} contain some auxiliary results involving
analytic function theory.
The theorems are proved in Section \ref{Sec:proof}.

Throughout the paper, $a_N\ll b_N$ means as $N\rightarrow\infty$, $a_N/b_N\rightarrow 0$, and $a_N\approx b_N$ means $a_N/b_N\rightarrow c$ for some positive constant $c>0$.
\section{How a transformation affects the Hermite rank}\label{Sec:affine}
First we state the results regarding the instability  of the Hermite rank  under transformation. Their relevance is explained in  Remark \ref{Rem:instab} which follows.  The proofs are given in Section \ref{Sec:proof}.

   Recall that $x$ is an \emph{accumulation point} or \emph{limit point} of a set $E\subset\mathbb{R}^p$,  if every neighborhood of $x$ contains an infinite number of elements of $E$.  The so-called \emph{derived set} $E'$ consists of all the accumulation points of $E$. For example, if $E=\{0\}\cup [1,2)$, then $ E'=[1,2]$.

\begin{Thm}[Instability with respect to  shift]\label{Thm:shift}
Suppose that the measurable function $G(\cdot):\mathbb{R}\rightarrow\mathbb{R}$ is  not a.e.\ constant  and $G(\cdot\pm M)\in L^2(\gamma)$ for some $M>0$. Then $G(\cdot +x)\in L^2$ for all $x\in[-M,M]$, and the  interval $(-M,M)$ contains no accumulation point of the set
 \[
 E=\{x\in (-M,M):  ~ G(\cdot+x) \text{ has Hermite rank   }  \ge 2\},
 \]
namely, $E'\cap (-M,M)=\emptyset$.
\end{Thm}

\begin{Thm}[Instability with respect to  scaling]\label{Thm:scale}
Suppose that the measurable function $G:\mathbb{R}\rightarrow \mathbb{R}$ is not a.e.\ symmetric  and    $G(\cdot\times M)\in L^2(\gamma)$ for some   $M>1$.    Then $G(\cdot \times y)\in L^2$ for all $y\in (0,M]$, and the interval $(0,M)$ contains no accumulation point of the set
 \[
 E=\{y\in (0,M):  ~ G(\cdot\times y) \text{ has Hermite rank   } \ge  2\},\]
namely,  $E'\cap (0,M)=\emptyset$.
\end{Thm}
\begin{Rem}\label{Rem:sym}
The requirement of $G$ being non-symmetric in Theorem \ref{Thm:scale} is essential because any non-zero symmetric function $f\in L^2(\gamma)$ has Hermite rank  $2$ since due to the symmetry of the measure $\gamma$
\[
\int_{\mathbb{R}} f(z)H_1(z) \gamma(dz)=\int_{\mathbb{R}} f(z)z \gamma(dz)=0,
\] while
\[
\int_{\mathbb{R}} f(z) (z^2-1) \gamma(dz)=
\int_{\mathbb{R}} f(z) H_2(z) \gamma(dz)=2\int_0^\infty f(z)z^2 \gamma(dz)>0
\] due to symmetry of $\gamma$.
\end{Rem}
\begin{Rem}\label{Rem:instab}
The preceding two theorems point to the \emph{instability of the Hermite rank}. They imply, for example, that the particular values $x=0$ in Theorem \ref{Thm:shift} or $y=1$ in Theorem \ref{Thm:scale} cannot be an accumulation point of the set $E$ with  Hermite rank greater or equal to $2$. Hence some neighborhood of $x=0$ or $y=1$ contains      at most one point ($x=0$ or $y=1$) with Hermite  rank $k\ge 2$. This means that a slight  level shift or scale change (when the original transformation is non-symmetric) will force the   rank to change from perhaps being $\ge 2$  to being  $1$. In the case of shift, this is stated as Theorem 2.14 in \citet{bai:taqqu:2017:instability}. Note that the rank $2$ of a non-constant symmetric function in Remark \ref{Rem:sym}  is still unstable with respect to a level shift by Theorem \ref{Thm:shift}.
\end{Rem}

In addition, one can consider the       shift and scaling  joint together, namely, deal with an affine transformation.
\begin{Thm}[Instability with respect to  affine transformation]\label{Thm:joint}
Suppose that the measurable $G(\cdot)$ is not constant a.e. and $G(\pm M_1+\cdot\times M_2)\in L^2(\gamma)$. Then the  set \[
E=\{(x,y)\in (-M_1,M_1)\times (0,M_2):  ~ G(x+ \cdot \times y ) \text{ has Hermite rank   }\ge   2\}\]
has Hausdorff dimension (see \citet{falconer:2004:fractal}, Section 2.2) not exceeding $1$. So in particular,  $E$ has $0$ two-dimensional Lebesgue measure.
\end{Thm}

To illustrate Theorem \ref{Thm:joint}, consider the function $G(z)=z^2=1+H_2(z)$ which has Hermite rank $2$. Suppose that it is perturbed by an affine transformation involving a shift $x$ and a scale $y$ and becomes
\[(x+zy)^2=x^2+2xyz+y^2z^2= x^2+y^2 + (2xy)z+y^2(z^2-1)=x^2+y^2 + (2xy)H_1(z)+y^2H_2(z).
\]
After centering,    the centered function in $z$ becomes $(2xy)H_1(z)+y^2H_2(z)$, which has Hermite rank $1$  if and only if $x\neq 0$ (it no longer has Hermite rank $2$ as $G(z)$ does) . The  set $E$ in Theorem \ref{Thm:joint} corresponds to the one-dimensional line $\{(x,y):~x=0\}$ which has Lebesgue measure zero.

 Finally, we can formulate an abstract result about the instability with respect to  general nonlinear transformations.
\begin{Thm}\label{Thm:nonlinear per}
Suppose that $G\in L^2(\gamma)$  and let $F_\theta$ be a family of transformations parameterized by  $\theta\in D\subset \mathbb{R}^p$, where $D$ is an open region containing the origin $\mathbf{0}=(0,\ldots,0)$ and let $F_\mathbf{0}$ be the identity transformation. Suppose that $G$ is perturbed and becomes $G\circ F_\theta$.

 Assume that  $G\circ F_\theta\in L^2(\phi)$ for all $\theta\in D$, $Z$ is a standard normal random variable, and
let
\[
U(\theta):= \E [  (G\circ F_\theta)(Z)\times H_1( Z)]=\E [ (G\circ F_\theta)(Z)\times Z  ].
\]
be the first coefficient of the Hermite expansion of $G\circ F_\theta$, and
  assume that the following properties hold:
\begin{enumerate}[(a)]
\item  $U(\theta)$ is real analytic in $\theta\in D$;
\item   If $U(\theta)=0$ holds for all $\theta\in D$, then $G$ is   constant a.e..
\end{enumerate}
 Then either $G$ is constant a.e., or if not, then the set
 \[E=\{\theta\in D: ~G\circ F_\theta \text{ has Hermite rank }\ge 2\}=\{\theta\in D:  ~U(\theta)=0\}\] has Hausdorff dimension  not exceeding $p-1$. In particular,  $E$ has zero $p$-dimensional Lebesgue measure.
\end{Thm}

\begin{Rem}
Here is a more informative description of the set $E$ in   Theorem \ref{Thm:joint} and \ref{Thm:nonlinear per} above.  First, $E$ is called the zero set of $U$ if it consists of all $\theta\in D$ such that $U(\theta)=0$. Second,       the zero set $E=\{\theta\in D:~U(\theta)=0\}$  of a (multivariate) real analytic function $U$ is called a real analytic variety. According to Lojasiewicz's Structure Theorem (Theorem 6.3.3 of \citet{krantz:parks:2002:primer}),  $E\subset \mathbb{R}^p$ can be expressed as a union of   submanifolds of dimensions $0,1,\ldots,p-1$. We state here the results in terms of the Hausdorff dimension   for simplicity.
\end{Rem}

\section{Being near a higher order Hermite rank}\label{Sec:near}
 In the so-called near integration analysis of unit root (e.g., \citet{phillips:1987:towards}), one studies the asymptotic behavior of an autoregressive model when the autoregression coefficient  tends to $1$ (unit root) as the sample size increases. Such analysis sheds lights on the situation where the coefficient is close to but not exactly $1$.
Here in a similar spirit, we carry out a ``near higher order rank''  analysis  for Theorem \ref{Thm:shift} regarding the limit theorem for the sum $\sum_{n=1}^N G(Y(n)+x_N)$, where the shift $x_N\rightarrow 0$  as the number of summands $N\rightarrow\infty$. The result is given in the following theorem,  where  $c,c_i$'s denote   constants whose value can change from line to line.
\begin{Thm}\label{Thm:near higher order}
Let $\{Y(n)\}$ be a  standardized stationary Gaussian process whose covariance satisfies (\ref{eq:cov LRD}) with Hurst index $H\in (1/2,1)$. Suppose that the function $G(\cdot)$ is not constant a.e.,     $G(\cdot\pm M)\in L^2(\gamma)$ for some $M>0$, and has an Hermite rank $k\ge 1$.  Set
\begin{equation}\label{e:s1}
S_N(x_N)=\sum_{n=1}^N\Big[G(Y(n)+x_N)-\E G(Y(n)+x_N)\Big],
\end{equation}
where $x_N\rightarrow 0$.

\noindent $\bullet$  Central limit case: suppose that $1/2<H<1-{1\over {2k}} $ (if $k=1$ this case does not exist).
Then as $N\rightarrow\infty$, $N^{(1/2-H)/(k-1)}\rightarrow 0$ and
\begin{flalign*}
    \begin{array}{L@{\quad}l@{{}\ConvFDD{}}l}
    (a)    if  $ x_N\ll N^{(1/2-H)/(k-1)}$:   &  N^{-1/2}S_{[Nt]}(x_N)   &  c B(t);   \\
     (b)  if  $x_N\approx N^{(1/2-H)/(k-1)}$:        &N^{-1/2}S_{[Nt]}(x_N)   & c_1 Z_{H,1} (t) + c_2 B (t); \\
     (c)  if  $ N^{(1/2-H)/(k-1)}\ll x_N \rightarrow 0$: & N^{-H} x_N^{1-k} S_{[Nt]}(x_N)   &  c  Z_{H,1} (t);
    \end{array}
\end{flalign*}
where the fractional Brownian motion $Z_{H,1}(t)$ and the Brownian motion $B(t)$ are independent.

\medskip
\noindent $\bullet$  Non-central limit case:  suppose that $1-{1\over {2k}}<H<1$.
 Let $H_G$ be
as in (\ref{eq:H_G general}).
Then as $N\rightarrow\infty$, $N^{H-1}\rightarrow 0$, and
\begin{flalign*}
 \begin{array}{L@{\quad}l@{{}\ConvFDD{}}l}
    (a)    if  $ 0<x_N\ll N^{H-1}$:   &  N^{-H_G}S_{[Nt]}(x_N) & c Z_{H,k} (t)  ;   \\
     (b)   if  $x_N\approx N ^{H-1}$: & N^{-H_G}S_{[Nt]}(x_N)&\sum_{m=1}^k  c_m Z_{H,m} (t); \\
     (c)  if  $N^{H-1}\ll x_N\rightarrow 0$:& N^{-H} x_N^{1-k} S_{[Nt]}(x_N)  & c  Z_{H,1} (t);
    \end{array}
\end{flalign*}
where the Hermite processes $Z_{H,m}(t)$'s  are defined through the same Brownian motion in (\ref{eq:Herm process}).
\end{Thm}

The terms "central limit case" and "non-central limit case" in the preceding theorem refer to the
 terminology used in Theorem \ref{Thm:nclt gaussian} and thus to the type of limits one obtains when there is no shift, that is when $x_N=0$.
\begin{Rem}
Theorem \ref {Thm:near higher order} has some interesting implications.
In the central limit case,     the critical order $N^{(1/2-H)/(k-1)}$  depends negatively on $H$. The larger $H$, the smaller the critical order. Thus the larger $H$ (but below $1-1/(2k)$) is, the more easily the effect of a higher-order rank   in the limit theorem gets surpassed by a shift perturbation. In  the non-central limit case, however, the relation is reversed since the critical order $N^{H-1}$ depends positively on $H$. Note that in the non-central limit case,  the order $N^{H-1}$ determining the border of the regimes  does not depend on the rank $k$
\end{Rem}

In Theorem \ref{Thm:near higher order non-center} below, we  provide a result  on  the \emph{non-centered} sum $\sum_{n=1}^N  G(Y(n)+x_N)$, assuming that $\E G(Y(n))=0$. It is expected that if $x_N$ tends to $0$ too slowly, then a deterministic trend will appear in the limit.

\begin{Thm}\label{Thm:near higher order non-center}
Under the assumptions of Theorem \ref{Thm:near higher order}, if in addition $\E G(Y(n))=0$ and $S_N$ is replaced by
\begin{equation}\label{e:s2}
\widetilde{S}_N(x_N)=\sum_{n=1}^N  G(Y(n)+x_N),
\end{equation}
then the following conclusions hold:

\noindent $\bullet$  Central limit case:    suppose that  $H<1-{1\over {2k}} $.
Then as $N\rightarrow\infty$:
\begin{flalign*}
 \begin{array}{L@{\quad}l@{{}\ConvFDD{}}l}
    (a) if $x_N\ll N^{-1/(2k)}$: & N^{-1/2}\widetilde{S}_{[Nt]}(x_N) & c B(t)\\
 (b) if $x_N\approx N^{1-1/(2k)}$:& N^{-1/2}\widetilde{S}_{[Nt]}(x_N)& c_1t+ c_2 B(t);   \\
     (c) if  $N^{1-1/(2k)}\ll x_N\rightarrow 0$: & N^{-1}x_N^{-k}\widetilde{S}_{[Nt]}(x_N)& ct ;
    \end{array}
\end{flalign*}
where $Z_{H,k}(t)$ and $B(t)$ are independent.

\medskip
\noindent $\bullet$  Non-central limit case:    suppose that  $H>1-{1\over {2k}}$.
 Let $H_G$ be
as in (\ref{eq:H_G general}).
Then as $N\rightarrow\infty$:
\begin{flalign*}
 \begin{array}{L@{\quad}l@{{}\ConvFDD{}}l}
    (a) if  $0<x_N\ll N^{H-1}$:  &N^{-H_G}\widetilde{S}_{[Nt]}(x_N) & c Z_{H,k} (t)  ;\\
 (b) if  $x_N\approx N ^{H-1}$: & N^{-H_G}\widetilde{S}_{[Nt]}(x_N) & c_0t+   \sum_{m=1}^k  c_m Z_{H,m} (t)  ;   \\
     (c) if  $N^{H-1}\ll x_N\rightarrow 0$:   & N^{-1} x_N^{-k} \widetilde{S}_{[Nt]}(x_N) & ct  ;
    \end{array}
\end{flalign*}
 where the Hermite processes $Z_{H,m}(t)$'s  are defined through the same Brownian motion in (\ref{eq:Herm process}).
\end{Thm}
%\begin{Rem}
%In regime (c) of both central and non-central limit cases above,  one may  consider subtracting the limit constant $c$ to obtain  non-degenerate limits with proper normalizations.  We omit this for   which can be done by examining the proofs. The centered  case   Theorem \ref{Thm:near higher order} might be more natural to consider  since the notion of Hermite rank involves centering (see (\ref{eq:def herm rank expan})).
%\end{Rem}

In the next theorem the argument $x_N$ in Theorem \ref{Thm:near higher order non-center} is replaced by a subtracted sample mean.
\begin{Thm}\label{Thm:near higher order centering}
Under the assumptions of Theorem \ref{Thm:near higher order non-center},   assume in addition that $G(\cdot)$ is a polynomial. Let $\bar{Y}_N=N^{-1}\sum_{n=1}^N Y(n)$. Then the following conclusions hold:

\noindent $\bullet$  Central limit case:    suppose that  $H<1-{1\over {2k}} $. Then as $N\rightarrow\infty$,
\[
N^{-1/2} \widetilde{S}_{[Nt]}(-\bar{Y}_N)\ConvFDD cB(t)
\]
\medskip
\noindent $\bullet$  Non-central limit case:    suppose that  $H>1-{1\over {2k}}$.
 Let $H_G$ be
as in (\ref{eq:H_G general}).
Then as $N\rightarrow\infty$,
\[
N^{-H_G}\widetilde{S}_{[Nt]}(-\bar{Y}_N)\ConvFDD  \sum_{m=1}^k c_m Z_{H,m}(t)
\]
where   the Hermite processes $Z_{H,m}(t)$ 's are defined through the same Brownian motion in (\ref{eq:Herm process}).
\end{Thm}
\begin{Rem}
As an example, one may take $G(z)=z^2$ in (\ref{e:s2}), which leads to the sample variance as considered in  \citet{hosking:1996:asymptotic} and \citet{dehling:taqqu:1991:bivariate}. See also Section 3.1 of \citet{bai:taqqu:2017:instability}. Comparing Theorem \ref{Thm:near higher order centering} with Theorem \ref{Thm:nclt gaussian}, centering $\{Y(n)\}$ by subtracting the sample mean may not affect the fluctuation order of the sum, but  can change the limit distribution.
\end{Rem}

%\begin{Rem}\label{Rem:near higher order}
%Some interesting observations can be made about the preceding theorems.
%\begin{enumerate}[(a)]
%\item  In the central limit case,     the longer the memory is ($H$ close to $1$), the more easily   the effect of a higher-order rank is erased from the shift, while in the non-central limit case, the relation is reversed.  In the non-central limit case,  the order $N^{H-1}$ determining the border of the regimes  does not depend on the rank $k$. *************more details*************
%\item　 The order $N^{H-1}$ is also the convergence rate of the sample mean of $\{Y(n),n=1,\ldots,N\}$, and $N^{H-1}\ll N^{(1/2-H)/(k-1)}$ if $H<1-\frac{1}{2k}$. So  in the case of a purely level-shift perturbation, centering (subtracting the sample mean) may  be recovered  correctly the fluctuation order of the sum predicted by a higher-order Hermite rank. However, the limit distribution may not recovered correctly.    below for an example. Nevertheless, the empirical study in Section \ref{Sec:empirical} suggests that for real-life data, centering may not recover the order predicted by a higher-order rank. *************more details*************
%\end{enumerate}
%\end{Rem}

The proofs of Theorem \ref{Thm:shift}, \ref{Thm:scale}, \ref{Thm:joint} and Theorems \ref{Thm:near higher order}, \ref{Thm:near higher order non-center} and
  \ref{Thm:near higher order centering} can  be found in   Section \ref{Sec:proof}, and are all based on the analysis of analytic functions.

\section{On the coincidence of Hermite and power ranks}\label{Sec:coin}

The Hermite rank has been defined in (\ref{eq:def herm rank}). We now define the \emph{power rank} which is used in the approach of \citet{ho:hsing:1997:limit} for limit theorems for transformations of long-memory  moving-average processes.
Given a function $G(\cdot)$ and a random variable $Y$ satisfying $\E G(Y)^2<\infty$, let
\begin{equation}\label{eq:G_infty}
G_\infty(y)=\E  G (Y+y)
\end{equation}
given that the expectation exists  and suppose that $G_\infty(\cdot)$ has   derivatives of  order sufficiently high. The power rank of $G(\cdot)$ with respect to $Y$ is defined as
\begin{equation}\label{eq:power rank}
\inf\{m \ge 1:~ G_\infty^{(m)}(0)\neq 0  \},
\end{equation}
where $G_\infty^{(m)}(y)$ denotes the $m$-th derivative of $G_\infty(y)$.
The following fact  was stated in \citet{ho:hsing:1997:limit} without a detailed proof. It was proved by \citet{levy-leduc:taqqu:2013:long} in the case where $G(\cdot)$ is a polynomial.
\begin{Pro}\label{Pro:coincide}
Suppose that $G(\cdot)\in L^2(\gamma)$. Then
the power rank in (\ref{eq:power rank}) coincides with the Hermite rank \eqref{eq:def herm rank} if $Y$ is Gaussian.
\end{Pro}
\begin{proof}%[Proof of Proposition \ref{Pro:coincide}]
Proposition \ref{Pro:coincide} is a direct consequence of (\ref{eq:ana expan G_infty}) below, namely,
\[ \sum_{m=0}^\infty \frac{G_\infty^{(m)}(0)}{m!} x^m =\sum_{m=0}^\infty \frac{\E[G(Z)H_m(Z)]}{m!} x^m,
\]
\end{proof}

\section{Auxiliary results}\label{Sec:aux}
We prove here some auxiliary results involving the Weierstrass transform and analytic function theory.

Let $Z$ denote throughout a standard Gaussian random variable and let $\phi=d\gamma/dx=(2\pi)^{-1/2}e^{-x^2/2}$ denote its density.  Define  the function
\begin{equation}\label{eq:G_infty}
 G_\infty(x)= \E  G (Z+x)=\int_{\mathbb{R}}{G}(z+x)\phi(z)dz=\int_{\mathbb{R}}{G}(z)\phi(z-x)dz,
\end{equation}
whenever the integrability holds.
The function $G$   may not be smooth, but the function $G_\infty$ is   smooth due to the convolution with the smooth $\phi(z)$.  $G_\infty$ is called the \emph{Weierstrass transform} of $G$ (see \citet{hirschman:widder:2005:convolution}, Chapter VIII).

First we state some preliminary   facts.
Recall that a function $f:\mathbb{R}^p\rightarrow \mathbb{R}$ is real analytic over an open domain $D\subset\mathbb{R}^p$, if at every point $y=(y_1,\ldots,y_p)\in D$, there exists a  neighborhood $B\subset D$ of $y$, such that
\[
f(x_1,\ldots,x_p)=\sum_{i_1,\ldots,i_p=1}^\infty a_{i_1,\ldots,i_p}  (x_1-y_1)^{i_1}\ldots (x_p-y_p)^{i_p}
\]
for some coefficient $a(i_1,\ldots,i_k)$, where the series converges absolutely in $B$.
It is well-known that $f$ is  infinitely differentiable, and common elementary operations including composition, affine transform, multiplication preserve  analyticity. See, e.g., \citet{krantz:parks:2002:primer}, Chapter 1.
\begin{Lem}\label{Lem:shift scale still int}~
\begin{enumerate}[(a)]
\item
If $g(\cdot+ a)$, $g(\cdot+ b) \in L^1(\phi)$, $a<b$, then
\begin{enumerate}
\item[(a1)]
$g(\cdot + x)\in L^1(\phi)$ for any $a<x<b$;
\item[(a2)]
 the Weierstrass transform $\E g(Z+x)$ is a real analytic function in $x\in (a,b)$.
\end{enumerate}
\item If $g(\cdot \times c)\in L^1(\phi)$, then $g(\cdot \times y)\in L^1(\phi)$ for any $0<y<c$.
\end{enumerate}
\end{Lem}
\begin{proof}

(a) The first statement (a1) follows from the fact that $\phi(z-x)\le  A\phi(z- a)+ B \phi(z- b)$ for some sufficiently large constants $A,B>0$ by exploring the shape  of $\phi$.   To obtain  statement (a2), first as in item 2.2 of \citet{hirschman:widder:2005:convolution} Chapter VIII, the Weierstrass transform $(\mathcal{W} g)(x):=\E g(Z+x)=(g\ast \phi)(x)$ can be expressed by  a bilateral Laplace transform  $(\mathcal{L}g)(x):=\int_{\mathbb{R}}e^{-zx}g(z)dz$  as
\begin{equation}\label{eq:laplace}
(\mathcal{W}g)(x)=\frac{1}{\sqrt{2\pi}}\int_{\mathbb{R}} f(z)  e^{-(x-z)^2/2}dz=\frac{e^{-x^2/2}}{\sqrt{2\pi}} \int_{\mathbb{R}}   e^{(-x)(-z)} [f(z) e^{-z^2/2}]dz=  \frac{e^{-x^2/2}}{\sqrt{2\pi}} (\mathcal{L} g)(-x),
\end{equation}
where $g(z)=f(z) e^{-z^2/2}$. Note that $ e^{-x^2/2}$ is analytic.
Then the analyticity   of $\mathcal{W} g$  follows from the analyticity of the bilateral Laplace transform  $\mathcal{L} g$. See \citet{widder:2010:laplace} Chapter VI, p.140. The statement
(b) follows from $\phi(x/y)\le C \phi(x/c)$ for some sufficiently large constant $C>0$, because $\int_{\mathbb{R}} |g(xc)|\phi(x)dx= c^{-1}\int_{\mathbb{R}} |g(x)|\phi(x/c)dx<\infty$
\end{proof}

\begin{Lem}\label{Lem:analytic}
Suppose $G\in L^2(\phi)$. Then $\E  |G (Z+x)|<\infty$, namely, $G(x+\cdot)\in L^1(\phi)$, for all $x\in\mathbb{R}$. Furthermore,  $G_\infty(x)$ in (\ref{eq:G_infty}) admits the analytic expansion
\begin{equation}\label{eq:ana expan G_infty}
G_\infty(x)=\sum_{m=0}^\infty \frac{G_\infty^{(m)}(0)}{m!} x^m =\sum_{m=0}^\infty c_m x^m,
\end{equation}
where $G_\infty^{(m)}$ is the $m$-th derivative of $G_\infty$ and \[c_m=\frac{\E[G(Z)H_m(Z)]}{m!}\] are the coefficients of the Hermite expansion   in (\ref{eq:Herm Coeff}).

If in addition, $G(\cdot \pm M)\in L^2(\phi)$, $M>0$, then
\begin{equation}\label{eq:G^m(x)}
 G_\infty^{(m)}(x)  =  \E [G(Z+x) H_m(Z)],\qquad |x|<M,  \quad m=0,1,2,\ldots.
\end{equation}
\end{Lem}
\begin{proof}
Let $H_m(x)$'s be the Hermite polynomials defined in \eqref{eq:Herm Poly}. Then  by the Cauchy-Schwartz inequality and Corollary 5.1.2 of \citet{pipiras:taqqu:2017:long},   one has for any $x\in \mathbb{R}$ that
\begin{equation}\label{eq:abs sum}
\sum_{m=1}^\infty \frac{|x|^m}{m!} \int_{\mathbb{R}}|{G}(z) H_m(z)|\phi(z) dz \le \sum_{m=1}^\infty \frac{|x|^m}{m!}  \|{G}\|_{L^2(\phi)} \|H_m\|_{L^2(\phi)}=\|{G}\|_{L^2(\phi)}\sum_{m=1}^\infty \frac{|x|^m}{\sqrt{m!}}  <\infty
\end{equation}
since $\|H_m\|_{L^2(\phi)}=\sqrt{m!}$.
By Proposition 1.4.2 of  \citet{nourdin:peccati:2012:normal},  one has
\[
\sum_{m=0}^\infty \frac{x^m}{m!}H_m(z)=\exp( -(z-x)^2/2 +z^2/2) .
\]
So
\[
\sum_{m=0}^\infty \frac{ x^m}{m!} \int_{\mathbb{R}} {G}(z) H_m(z) \phi(z) dz =\int_{\mathbb{R}}{G}(z)   e^{ -(z-x)^2/2 +z^2/2} \phi(z)  dz=\int_{\mathbb{R}}{G}(z)   \phi(z-x)  dz=G_\infty(x),
\]
where the change of the order between the sum and the integral in the first equality  can be justified by (\ref{eq:abs sum}) and Fubini's Theorem.
This shows that $G_\infty$   has the desired expansion (\ref{eq:ana expan G_infty}).  Note that the same computation as above with $G$ replaced by $|G|$   shows that $\E  |G (Z+x)|=\int_{\mathbb{R}}|{G}(z)|   \phi(z-x)  dz<\infty$ by (\ref{eq:abs sum}).

The formula (\ref{eq:G^m(x)}) follows from
\[
 G_\infty^{(m)}(x)=\frac{d}{dx^m}\int_{\mathbb{R}} G(z) \phi(z-x)dz =\int_{\mathbb{R}} G(z) \frac{d}{dx^m}\phi(z-x)dz= \int G(z)H_m(z-x)\phi(z-x)dx
\]
by (\ref{eq:Herm Poly}), where the differentiation under the integral can be justified by the Dominated Convergence Theorem and the Mean Value Theorem as in the proof of Lemma \ref{Lem:biv ana} below.
\end{proof}

\begin{Lem}[Uniqueness of  Weierstrass transform]\label{Lem:Gaussian vanish}
If $f\in L^1(\phi)$ and the convolution $f\ast \phi=0$ a.e., then $f=0$ a.e..
\end{Lem}
\begin{proof}
First, express as in (\ref{eq:laplace}) the Weierstrass transform  by a bilateral Laplace transform. Then the conclusion follows from the uniqueness of the bilateral Laplace transform. See \citet{widder:2010:laplace} Chapter VI, Theorem 6b.
\end{proof}

\begin{Lem}\label{Lem:scale ana}
Suppose that $G(\cdot\times M)\in L^2(\phi)$ for some $M\in (1,+\infty)$. Then $\E |G(yZ)Z|<\infty$ and
\[F(y):=\E G(yZ)Z
\] is an analytic function in  $ y\in (0,M)$.
\end{Lem}
\begin{proof}
Since $\phi(t/y)\le c \phi (t/M)$ for some constant $c>0$,
\[
\E |G(yZ)Z|=\int_\mathbb{R} |G(y  z)z|\phi(z)dz=\frac{1}{y^2}\int_\mathbb{R} |G(t)t|\phi(t/y)dt \le C\int_\mathbb{R} |G(M  z)z|\phi(z)dz,
\]
and hence by Cauchy-Schwartz,
\[
\E |G(yZ)Z|\le  C \|G(\cdot \times M)\|_{L^2(\phi)} <\infty
\]
for some constant $C>0$   depending only on $y$.
 Next
applying the change of variable $u=(yz)^2$ on $z\ge 0$ and on $z<0$, we have
\begin{equation}\label{eq:H scale}
F(y)=\E G(yZ)Z= \int_{\mathbb{R}}  zG(yz) \frac{e^{-z^2/2}}{\sqrt{2\pi}} dz = \frac{1}{2y^2\sqrt{2\pi}} \left[ \int_{0}^{\infty} [G(u^{1/2})-G(-u^{1/2})] \exp[-u  (2y^2)^{-1}] du   \right].
\end{equation}
The   integral  in (\ref{eq:H scale}) is a   Laplace transform  evaluated  at $(2y^2)^{-1}$.  The Laplace transform is real analytic over the the open half interval where the integrability holds (\citet{widder:2010:laplace} Chapter II, Theorem 5a). The  function $y^{-2}$ is also analytic in $y\in (0,M)$. The conclusion then  follows since elementary operations preserve analyticity.
\end{proof}

\begin{Lem}\label{Lem:biv ana}
Let $D=(-M_1,M_1)\times (0,M_2)$, $M_1>0$, $M_2>1$. Suppose that $g$ is a measurable function such that $ g(\pm M_1+ \cdot\times M_2)   \in L^2(\phi)$. Then the function
\begin{equation*}
f(x,y):=    \int_{\mathbb{R}} g(u)  \exp\left[  -\frac{(u-x)^2}{2 y^2} \right]du
\end{equation*}
is    real analytic over $(x,y)\in D$.
\end{Lem}
\begin{proof}
The assumption implies that
\begin{equation}\label{eq:in L2}
g(x+\cdot\times y )   \in L^2(\phi)
\end{equation}  for all $(x,y)\in D$ by Lemma \ref{Lem:shift scale still int}.
By Theorem 1 of \citet{siciak:1970:characterization},  one needs to show that
\begin{enumerate}[(1)]
\item   $f$ is   $C^\infty$ (infinitely smooth);
\item $f$ is univariately analytic along any direction at an arbitrary point $(x_0,y_0)\in D$.
\end{enumerate}

For  part (1),  we need to check that the partial derivatives of all orders of $f(x,y)$ exist and are continuous. First, the   derivatives with respect to $x$ and $y$ can be taken under the integral sign  by applying the Dominated Convergence Theorem and the Mean Value Theorem  with the help of the following facts which are justified afterwards:
\begin{enumerate}
\item For any integers $i,j \ge 0$,
\[\frac{\partial^{i+j}}{\partial x^i \partial y^j} \exp( -\frac{(u-x)^2}{2 y^2}) =Q(u,x,y^{-1})\exp( -\frac{(u-x)^2}{2 y^2})\]
for some trivariate polynomial $Q(x_1,x_2,x_2)$.
\item  For any $x_0\in \mathbb{R} $ and  small $\delta>0$, there exists a polynomial $Q^*$ with non-negative coefficient  such that
 \[\sup_{|x-x_0|\le \delta}|Q(u,x,y^{-1})|\le Q^*(|u|,|y|^{-1});\]
  A similar bound holds for $\sup_{|y-y_0|\le \delta}|Q(u,x,y^{-1})|$.
\item If $\delta$ is chosen sufficiently small so that $\exp(-\frac{4\delta^2}{y^2})\ge 1/2,$
 then
 \[\sup_{|x-x_0|\le \delta}\exp( -\frac{(u-x)^2}{2 y^2})\le \exp( -\frac{(u-x_0+\delta)^2}{2 y^2})+\exp( -\frac{(u-x_0-\delta)^2}{2 y^2}),\] and for any $\delta>0$, \[\sup_{|y-y_0|\le \delta}\exp( -\frac{(u-x+\delta)^2}{2 y^2})=\exp( -\frac{(u-x+\delta)^2}{2 (y_0+\delta)^2}).\]
\item  For any integer $n\ge 0$ and polynomial $Q$, \[\int |g(u)| Q(|u|) \exp( -\frac{(u-x)^2}{2 y^2}) du<\infty\]  for all $(x,y)\in D$.
\end{enumerate}

Item 1 can be verified easily by induction and the elementary differentiation rules.

To see Item 2, write the polynomial as $Q(u,x,v)= a_0(x)+a_1(x)u+a_2(x)v+\ldots $. So $|Q(u,x,v)|\le |a_0(x)|+|a_1(x)u|+|a_2(x)v|+\ldots $.
Note that each coefficient $|a_i(x)|$ is a continuous function. Over a  compact interval $[x_0-\delta,x_0+\delta]$, each $|a_i(x)|$  is bounded   by a positive constant.  Replacing $|a_i(x)|$'s by these constants yields the conclusion.

For Item 3, the maximum of the function $\exp( -\frac{(u-x)^2}{2 y^2})$ in $x$ over the interval $[x_0-\delta,x_0+\delta]$ can   be either $1$ if  $u \in [x_0-\delta,x_0+\delta]$, or $\exp( -\frac{(u-x_0\pm \delta)^2}{2 y^2})$ otherwise. In the case where $u \in [x_0-\delta,x_0+\delta]$,   by the choice of $\delta$ described in Item 3 and  $|u-x_0\pm \delta|\le 2 \delta $, we have  \[\exp( -\frac{(u-x_0+\delta)^2}{2 y^2})+\exp( -\frac{(u-x_0-\delta)^2}{2 y^2})\ge1/2+1/2\ge  1;\]
 in the case where $u \notin [x_0-\delta,x_0+\delta]$, e.g., if $u<x_0-\delta$, then one has $\exp( -\frac{(u-x_0)^2}{2 y^2})\le \exp( -\frac{(u-x_0+\delta)^2}{2 y^2})$ by monotonicity. The second inequality follows easily from monotonicity.

To obtain Item 4, by Cauchy-Schwartz and (\ref{eq:in L2}), we have
\[
\int |g(u)| Q(|u|) \exp( -\frac{(u-x)^2}{2 y^2}) du\le y(2\pi)^{1/2} [\E g(x+yZ)^2]^{1/2} [\E Q(|x+yZ|)^2 ]^{1/2}<\infty
\]

Continuity of $f(x,y)$ can as well be verified  using Dominated Convergence and the facts 2$\sim$4 above.

 To check  claim (2), set $x(t)=x_0+at$, $y=y_0+bt$ where $(a,b)\in \mathbb{R}^2$ satisfying $a^2+b^2=1$.
Suppose that  $(x(t),y(t))\in D$
when  $|t|<\delta$ , where  $\delta>0$ will be adjusted smaller if necessary  later. We want to check the real analyticity of
\begin{equation}\label{eq:f(t)}
f(t):=f(x(t),y(t))=\int_{\mathbb{R}} g(u)  \exp\left[  -\frac{(u-x_0-at)^2}{2 (y_0+bt)^2} \right]du.
\end{equation}
at $t=0$.
Note that the function
\[h(u,z):= \exp(A(u,z)):=\exp\left[  -\frac{(u-x_0-az)^2}{2 (y_0+bz)^2} \right]
\] is  complex analytic in $z\in B_\delta:=\{z\in \mathbb{C}: |z|<\delta\}$ for each fixed $u\in \mathbb{R}$ since     $y_0+bz\neq 0$.   Next, by choosing $\delta$ small enough, one can ensure that for some  constant $c,\mu>0$  sufficiently large  and  $\sigma\in (0,M_2)$, such that
\begin{align}
\mathrm{Re}[A(u,z)]& =  -  \frac{\mathrm{Re}[ (u-x_0-az)^2(y_0+b\bar{z})^2 ]}{2| y_0+bz|^4 }= \frac{- y_0^2 u^2  +\mathrm{Re}[\ldots] }{2| y_0 +bz |^4 } = \frac{-  u^2  +\mathrm{Re}[\ldots] }{2| y_0^{1/2}+bz/y_0^{1/2}|^4 }\label{eq:Re}.
\end{align}
Since $y_0\in (0,M_2)$ and $|z|<\delta$ where $\delta$ is sufficiently small, there exists $\sigma\in (0,M_2)$, so that\begin{equation}\label{eq:den bound}
| y_0^{1/2}+bz/y_0^{1/2}|^4\ge \sigma^2
\end{equation} for all $|z|<\delta$. On the other hand,  the bracket $[\ldots]$ in (\ref{eq:Re}) is of the form
\[
c_1(z,\bar{z},x_0)u+c_0(z,\bar{z},x_0),
\]
where $c_1$ and $c_0$ are polynomials. For all $|z|<\delta$, one has  for some large $\mu,c>0$ that
\begin{equation}\label{eq:Re bound}
\mathrm{Re}[\ldots]\le  |c_1(z,\bar{z},x_0)||u|+|c_0(z,\bar{z},x_0)|\le 2\mu |u|-\mu^2  + 2\sigma^2 c
\end{equation}
Combining (\ref{eq:Re}), (\ref{eq:den bound}) and (\ref{eq:Re bound}), one has
\[
\mathrm{Re}[A(u,z)]\le  \frac{-   u ^2 + 2\mu |u|-\mu^2}{2\sigma^2}+c
\]
Hence for some constant $C>0$,
\begin{equation}\label{eq:bound}
|h(u,z)| = \exp(\mathrm{Re}[A(u,z)]             )\le C[  \exp \left( -  (u^2+\mu)^2/(2\sigma^2)\right) +\exp\left(-  (u^2-\mu)^2/(2\sigma^2)\right) ]  ,\quad z\in B_\delta.
\end{equation}
In   $\int_{\mathbb{R}}| g(u) h(u,z)| du$,
replace $|h(u,z)|$ by the preceding bound (\ref{eq:bound}), and note that $g(\pm\mu+\cdot\times \sigma)   \in L^1(\phi)$ for any $\mu\in \mathbb{R}$ and $\sigma\in (0,M)$   by Lemma \ref{Lem:shift scale still int} (b) and Lemma \ref{Lem:analytic}. Then one deduces that
\[
\sup_{z\in B_\delta}\int_{\mathbb{R}}| g(u) h(u,z)| du<\infty.
\]
  This  fact with Fubini's Theorem justifies the following order change of integrals:
\[
\oint_{\Delta}f(z)dz=\oint_{\Delta}dz\int_{\mathbb{R}} g(u)h(u,z)du  = \int_{\mathbb{R}} g(u) du\oint_{\Delta} h(u,z)  dz =0,
\]
where $\Delta$ is a closed triangle within $B_\delta$, and the last equality is due to  Cauchy's theorem since $ h(u,z)$ is complex analytic in $z\in B_\delta$. Then the complex analyticity of $f(z)$ over ${B}_\delta$ follows from Morera's Theorem (Theorem 10.17 of \citet{rudin:1987:real}). Restricting $B_\delta$ to $\mathrm{Im}(z)=0$ yields the real analyticity.
\end{proof}

%%%%%%%%%%%%%%%%%%%%%%%%%%%%%%%%%%%%%%%%%%%%%%%%
\section{Proof of the theorems}\label{Sec:proof}
We prove here Theorems \ref{Thm:shift}, \ref{Thm:scale}, \ref{Thm:joint} involving the instability of the Hermite rank and
and Theorems \ref{Thm:near higher order}, \ref{Thm:near higher order non-center} and   \ref{Thm:near higher order centering} involving "near higher order rank".

\begin{proof}[Proof of Theorem \ref{Thm:shift}] Suppose that the function G is not a.e.\ constant.
Define $E=\{x\in (-M,M):~ G(\cdot+x) \text{ has Hermite rank} \ge 2 \}$. Note first that we have
$E= \{x\in (-M,M): ~G_\infty^{(1)}(x)=0\}$ in view of Lemma \ref{Lem:analytic} applied to $G(\cdot+x)$. Indeed, since the Hermite rank of $G(\cdot+x)$ is greater than $1$, we have $G_\infty^{(1)}(x)=\E G(Z+x)H_1(Z)=0$.
Now suppose by contradiction that $E$ has an accumulation point in $(-M,M)$.
Since $G_\infty$ is analytic by Lemma \ref{Lem:analytic}, so is the derivative $G_\infty^{(1)}(x)$. By \citet{rudin:1976:principles} Theorem 8.5,  $G_\infty^{(1)}$ is identically zero on $(-M,M)$ and hence $G_\infty$ is a constant on $(-M,M)$. But by (\ref{eq:G_infty}), $G_\infty=G\ast \phi$. By Lemma \ref{Lem:Gaussian vanish} and linearity, ${G}$ is a.e.\ a constant as well  which contradicts the theorem assumption.
\end{proof}

\begin{proof}[Proof of Theorem \ref{Thm:scale}]
Suppose that the function G is not a.e.\ symmetric.
Now the set with Hermite rank $\ge 2$  is $E=\{y\in (0,M): ~F(y):=\E G(yZ)Z=0\}$.
 The proof is similar to that of Theorem \ref{Thm:shift} above. Suppose by contradiction that $E$ has an accumulation point in $(0,M)$. Since  $F$ is analytic on $(0,M)$ by Lemma \ref{Lem:scale ana}, so $F(y)$ is identically zero on $(0,M)$.
Then we apply  the uniqueness of the Laplace transform  (\citet{widder:2010:laplace}  Chapter II, Theorem 6.3)
to relation (\ref{eq:H scale}) to conclude that $G(z)=G(-z)$ a.e., which contradicts   the assumption.
\end{proof}

%\begin{proof}[Proof of Corollary \ref{Cor:shift}]
%We prove the case of shift, and the case of scaling is similar.
%By Theorem \ref{Thm:shift}, we have in particular that  $0$ cannot be an accumulation point of
%$E$. This implies that there is a neighborhood $ (-\epsilon, \epsilon) $ of $0$ such that
% $E \cap (-\epsilon, \epsilon)$ contains at most finitely many points. By decreasing this neighborhood if
%necessary, one can further ensure that  it contains at most the single point $0$. Therefore the rank is 1 on $(-\epsilon, 0) \cup (0,\epsilon) $.
%\end{proof}

\begin{proof}[Proof of Theorem \ref{Thm:joint}]
Define
\[F(x,y)=\E[ Z G(x+yZ)].
\] As before $E=\{(x,y)\in D: ~F(x,y)=0 \}$.
By a change of variable $z= ( u-x)/y$ we can write
\begin{equation}\label{eq:F(x,y)}
F(x,y)= \frac{1}{y^2\sqrt{2\pi}} \int_{\mathbb{R}}( u-x)G(u) \exp\left[  -\frac{(u-x)^2}{2 y^2} \right] du.
\end{equation}
Then $F(x,y)$ is a bivariate real analytic function by Lemma \ref{Lem:biv ana} since
$ (\pm M_1 + M_2 u) G(\pm M_1+M_2 u )\in L^2(\phi)$ by Cauchy-Schwartz.  If $F(x,y)\equiv 0$, which implies $F(x,1)\equiv 0$, then the Hermite rank of $G(\cdot+x)$ is greater than $1$ for all $x\in (-M_1,M_1)$, which contradicts Theorem \ref{Thm:shift}.
 So $F(x,y)$ is not identically zero.
The claimed properties of $E$, the zero set of the analytic $F$, then follow from \citet{mityagin:2015:zero} (see also \citet{krantz:parks:2002:primer}, Section 4.1).

\end{proof}

\begin{proof}[Proof of Theorem \ref{Thm:nonlinear per}]
If $G$ is not a constant a.e., by assumption  (b), the analytic $U(\theta)$ is not identically zero on $D$. So taking into account assumption (a),  the zero set $E$ has the claimed properties by \citet{mityagin:2015:zero}.
\end{proof}

\begin{proof}[Proof of Theorem \ref{Thm:near higher order}]

First by the Hermite expansion  (\ref{eq:def herm rank expan}) we can write:
\begin{equation}\label{e:G2}
G(Z+x_N)=\sum_{m=0}^\infty c_m(x_N)  H_m(Z),
\end{equation}
 We set as in (\ref{eq:G_infty}) that $G_\infty(x_N)=\E G(Z+x_N)$, where $Z\sim N(0,1)$. In view of Lemma \ref{Lem:analytic}, we have
\[
c_m(x_N)=\frac{1}{m!}\E  G(Z+x_N)H_m(Z) =\frac{1}{m!} G_\infty^{(m)}(x_N).
\]
and
\begin{equation}\label{eq:H_infty}
F_\infty(x_N):= \E G(Z+x_N)^2=\sum_{m=0}^\infty c_m(x_N)^2 \E H_m(Z)^2=\sum_{m=0}^\infty c_m(x_N)^2m!.
\end{equation}

Since $G$ has Hermite rank   $k$, in view of Lemma \ref{Lem:analytic} we have $G_\infty^{(1)}(0)=\ldots=G_\infty^{k-1}(0)=0\neq G_\infty^{k}(0)$. By the analytic expansion,
\[
G_\infty(x_N)=\frac{G_\infty^{(k)}(0)}{k!} x_N^k+O(x_N^{k+1})
\]
and
\[
\quad G_\infty^{(m)}(x_N) =\frac{d^mG}{dx^m}  (x_N)  =\frac{1}{(k-m)!} g_k x_N^{k-m} +O(x_N^{k-m+1}), \quad m=1,\ldots,k ,
\]
where
\[g_k=G_\infty^{(k)}(0).\]
 Then
\begin{align}\label{eq:S_{[Nt]}}
S_{[Nt]}(x_N)&= \sum_{n=1}^{[Nt]}\Big[G(Y(n)+x_N)-\E G(Y(n)+x_N)\Big]\notag\\
&=A_N(t)+B_N(t)+C_N(t)+D_N(t) +H.O.T.=A_N(t)+B_N(t)+R_N(t) +H.O.T.
\end{align}
In (\ref{eq:S_{[Nt]}}),
\begin{align}
&A_N(t)= (x_N^{k-1}N^{H})  \frac{g_k}{(k-1)!} \left[\frac{1}{N^H}\sum_{n=1}^{[Nt]}  Y(n) \right], ~~ B_N(t)=\sum_{m=2}^{k-1} (x_N^{k-m} N^{H[m]}) \frac{g_k}{(k-m)!}  \left[\frac{1}{N^{H[m]}}\sum_{n=1}^{[Nt]}  H_k(Y(n))\right]\notag\\
& C_N(t)=  N^{H[k]} \frac{g_k }{k!}\left[ \frac{1}{N^{H[k]}} \sum_{n=1}^{[Nt]}  H_k (Y(n)) \right],~~ D_N(t)=\sum_{m=k+1}^\infty c_m(x_N)  \sum_{n=1}^{[Nt]}   H_m(Y(n)),~~R_N(t)=C_N(t)+D_N(t),\label{eq:A-D}
\end{align}
where
\begin{equation}\label{eq:H(m)}
H[m]=\max\Big((H-1)m+1,~ 1/2\Big),
\end{equation}
 $\sum_{m=a}^b\ldots$ is understood as zero if $a>b$, and
  $H.O.T.$ stands for a higher-order term that will be asymptotically negligible as $N\rightarrow\infty$ after suitable normalization. In  (\ref{eq:A-D}), the expressions within a bracket  $[ ~\cdot ~]$ converge  in distribution and  their variances also converge. See, e.g., \citet{pipiras:taqqu:2017:long}, Chapter 5. The term $A_N(t)$ involves the first Hermite polynomial since $H_1(Y)=Y$; the term $B_N(t)$ involves the second up to the $(k-1)$th Hermite polynomial; the term $C_N(t)$ involves the $k$th Hermite polynomial; the term $D_N(t)$  involves all the Hermite polynomials of orders higher than $k$. $R_N(t)$ involves the terms $m\ge k$ in (\ref{e:G2}).

 We assume  for simplicity that $H\neq 1/(2m)$, $m=1,\ldots,k$. Otherwise the proof undergoes a slight modification involving an logarithmic factor.
 We have to consider the behavior of $ \sum_{n=1}^{[Nt]}   H_m(Y(n))$ for all $ m\geq 1 $.
  Let
\begin{equation}\label{eq:k_0}
k_0=\sup\{m\in  \mathbb{Z}_+:~ H>  1-1/(2m)\}.
\end{equation}

%Observe first  the following relations:
%\begin{enumerate}[(1)]
%\item if $1\le m<n\le k_0 $, then $x_N^{k-m}N^{(H-1)m+1} \ll x_N^{k-n} N^{(H-1)n+1} \iff    x_N\ll N^{H-1}$  (also for $\approx$ and $\gg$);
%\item if $m<k$, then $x_N^{k-m}N^{1/2}\ll N^{1/2}$.
%\item if $x_N\ll N^{H-1}$ or $x_N\approx N^{H-1}$, $m<k$, then $x_N^{k-m} N^{(H-1)m+1}\ll N^{1/2}$ using  $H<1-1/(2k)$.
%\end{enumerate}
\noindent
In view of (\ref{eq:k_0}), when $m>k_0$ the central limit theorem in Theorem \ref{Thm:nclt gaussian} holds and when $m<k_0$ it is the non-central limit theorem that holds. Alternatively,
the central limit theorem  holds when
$H<1-1/(2k)$ and the non-central limit theorem holds when $H>1-1/(2k)$.

\bigskip
\noindent $\bullet$ Consider first  the central limit case where $H<1-1/(2k)$.

 In this case, $A_N(t)$ is associated with the order $x_N^{k-1} N^{H}$ and  $R_N(t)$ is   associated with the order $N^{1/2}$.
We now focus on $B_N(t)$ and claim that it is asymptotically negligible compared with either $A_N(t)$ or $R_N(t)$, namely,
\begin{equation}\label{eq:red clt}
x_N^{k-m} N^{H[m]}\ll \max( x_N^{k-1} N^{H},  N^{1/2}),
\quad m=2,\ldots,k-1.
\end{equation}
Indeed, when $m>k_0$, this is obvious since $H[m]=1/2$; Suppose now $m\le k_0$. In this case, $H[m]=(H-1)m+1$. However, in cases (a) and (b), where $x_N\ll$ or $\approx N^{(1/2-H)/(k-1)}$, the factor $x_N^{k-m}N^{H[m]}$ has an exponent bounded by
\begin{equation}
\frac{(1/2-H)(k-m)}{k-1}+(H-1)m+1=\frac{(m-1)k}{k-1} H +\frac{k+m-2km}{2(k-1)}+1<\frac{1}{2},
\end{equation}
where the last inequality can be obtained by plugging in the inequality $H<1-1/(2k)$ and some elementary computation. Hence in cases (a) and (b),  the relation (\ref{eq:red clt}) is proved.
In the case (c) where  $x_N\gg  N^{(1/2-H)/(k-1)}$,  it can be checked that $x_N^{k-m} N^{(H-1)m+1}\ll  x_N^{k-1} N^{H} $ using the inequality $\frac{1/2-H}{k-1}>H-1$, and the latter inequality follows from $H<1-1/(2k)$.
We have thus proved (\ref{eq:red clt}) and hence $B_N(t)$ is asymptotically negligible. We are left with $A_N(t)$ and $R_N(t)$.

The different asymptotic regimes  in the central limit case in Theorem \ref{Thm:near higher order} come about when comparing the order $x_N^{k-1} N^H$ of $A_N(t)$ with the order $N^{1/2}$ of $R_N(t)=C_N(t)+D_N(t)$. But we also have to deal with the convergence of the processes.
By \citet{bai:taqqu:2013:1-multivariate},  we have the joint convergence
\begin{equation}\label{eq:CLT joint}
 \left(\frac{1}{N^{H}}\sum_{n=1}^{[Nt]} Y(n), \frac{1}{N^{1/2}}\sum_{n=1}^{[Nt]}  \sum_{m=k}^\infty \frac{g_m}{m!} H_m(Y(n)) \right) \ConvFDD  \left(c_1 Z_{H,1}(t), c_2 B(t) \right)
\end{equation}
where the fractional Brownian motion $Z_{H,1}(t)$ and the Brownian motion $B(t)$ are independent. From (\ref{eq:CLT joint}) we set
\[
R_N'(t)=\sum_{n=1}^{[Nt]}  \sum_{m=k}^\infty \frac{g_m}{m!} H_m(Y(n)),
\]
which is $R_N(t)$ with $x_N=0$.

We cannot use (\ref{eq:CLT joint}) directly because we have $R_N(t)$ instead of $R_N'(t)$. We thus need to compare them first.
 Since $Y(n)$ is standardized, we have  $|\gamma_Y(n)|^{m}\le |\gamma_Y(n)|^{k}$ if $m\ge k$, where $\gamma_Y(n)=\Cov[Y(n),Y(0)]$. Thus  by a computation similar to those on p.299 of \citet{pipiras:taqqu:2017:long} using the orthogonality of the Hermite polynomials,
\begin{align}\label{eq:diff D B}
 \E |R_N(t)-R_N'(t)|^2 &=  \sum_{m>K}  m! (c_m(x_N) -c_m(0))^2 \sum_{|n|<[Nt]} ([Nt]-|n|) \gamma_Y(n)^m\notag\\
 & \le (Nt)\left(\sum_{n=-\infty}^\infty  |\gamma_Y(n)|^k\right)  \sum_{m=k}^\infty  m! (c_m(x_N)-c_m(0))^2.
\end{align}
Note that  for  a $K>k$,
\begin{align}\label{eq:near int 1}
&\sum_{m=k}^\infty  m! (c_m(x_N)-c_m(0))^2\le \sum_{m=k}^K  m! (c_m(x_N)-c_m(0))^2+ 2\sum_{m>K}  m! ( c_m(x_N)^2+ c_m(0)^2).
\end{align}
Recall $F_\infty(x)=\sum_{m=0}^\infty m! c_m(x)^2$  in (\ref{eq:H_infty}), which is analytic in view of Lemma \ref{Lem:shift scale still int}.
Hence
\[
\sum_{m>K}  m!  c_m(x_N)^2 = F_\infty(x_N)-F_\infty(0)+F_\infty(0)- \sum_{m=0}^K  m!   c_m(x_N)^2.
\]
As $N\rightarrow\infty$, by the continuity of $F_\infty$ and $c_m(x)=G^{(m)}_\infty(x)/m!$, we have
\begin{equation}\label{eq:near int 2}
\sum_{m>K}  m!  c_m(x_N)^2
\rightarrow  \sum_{m>K} m! c_m(0)^2.
\end{equation}
Take the limit $N\rightarrow\infty$ in  (\ref{eq:near int 1}) using (\ref{eq:near int 2}) and the continuity of $c_m$,  the right-hand side of (\ref{eq:near int 1}) is only left with the term $2\sum_{m>K} m! c_m(0)^2$, which tends to $0$  as $K\rightarrow\infty$. Combining this with (\ref{eq:diff D B}), we have
\begin{equation}\label{eq:approx clt}
\lim_{N\rightarrow\infty} N^{-1} \E | R_N(t)-R_N'(t)|^2= 0.
\end{equation}
It thus follows from (\ref{eq:CLT joint}) that
\[
\left(\frac{1}{N^{H}}\sum_{n=1}^{[Nt]} Y(n), \frac{1}{N^{1/2}}\sum_{n=1}^{[Nt]}  R_N(t) \right) \ConvFDD  \left(c_1 Z_{H,1}(t), c_2 B(t) \right).
\]
In case (a), we have $x_N^{k-1} N^{H}\ll N^{1/2}$ and thus as $N\rightarrow\infty$, $A_N(t)$ is negligible compared to $R_N(t)$ and thus we get $N^{-1/2}R_N(t)\ConvFDD c_2 B(t)$. In case (c), $x_N^{k-1}N^H \gg N^{1/2}$ and  the limit is $c_1 Z_{H,1}(t)$. In case (b), where $x_N^{k-1}N^H \approx N^{1/2}$, both limits appear.
Finally, note that $N^{-H} x_N^{1-k}  A_N(t)$ is proportional to the first component on the left-hand side  of (\ref{eq:CLT joint}), while by (\ref{eq:approx clt}), the term $N^{-1/2}R_N(t)$  can be asymptotically replaced by $N^{-1/2}R_N'(t)$, which is  the second component  on the left-hand side of (\ref{eq:CLT joint}). The rest of the proof in the central limit case can be carried out easily.

\medskip

\noindent $\bullet$ Now we consider the non-central limit case  where $H>1-\frac{1}{2k}$.
  We have to study, as before, the behavior of $ \sum_{n=1}^{[Nt]}   H_m(Y(n))$ for all $m\geq 1$.
 %  Note that   $k_0\ge k$  in (\ref{eq:k_0})  in this case.

  We first consider the terms $A_N(t)$, $B_N(t)$ and $C_N(t)$ which involve $m \leq k$.
In this case, we have the following relation: any $n,m \in \mathbb{Z}_+$,
\[x_N^{k-m}N^{(H-1)m+1} \ll x_N^{k-n} N^{(H-1)n+1} \iff    x_N\ll N^{H-1},\quad \text{which holds with $\ll$ replaced by $\approx$ or $\gg$ as well.}
\]
Hence  in case (a),   the term $A_N(t)$ in (\ref{eq:S_{[Nt]}}) contributes; in case (b), the terms $A_N(t)$, $B_N(t)$ and $C_N(t)$ all contribute; in case (c),  the term  $C_N(t)$ contributes.

We shall show below that the term $D_N(t)$, which involves $m>k$, is  negligible.
Set
$\gamma_Y(n)=\Cov(Y(n),Y(0))\approx n^{2H-2}$.  By a similar computation as in (\ref{eq:diff D B}),
\begin{equation}\label{eq:E D_N^2}
\E D_N(t)^2 = \sum_{m>k}  m! c_m(x_N)^2 \sum_{|n|<[Nt]} ([Nt]-|n|) \gamma_Y(n)^m .
\end{equation}
For an arbitrarily small $\epsilon>0$,  we have  for $m>k$ that
\begin{equation}\label{eq:bound nclt}
\left|\sum_{|n|<[Nt]} ([Nt]-|n|) \gamma_Y(n)^m \right|\le  c_1 (Nt) \sum_{|n|<[Nt]}  n^{(2H-2)(k+1)}  \le c_2 (Nt)^{2H[k+1]},
\end{equation}
for some constants $c_1,c_2>0$.
where $H[m]$ is as in (\ref{eq:H(m)}). Since $H[k+1] <H[k]$ when $H>1-1/(2k)$,  in view of (\ref{eq:E D_N^2}),  (\ref{eq:bound nclt}) and (\ref{eq:H_infty}),
\[
\E D_N(t)^2\le  c_2 (Nt)^{2H[k+1]} F_\infty(x_N)\ll N^{2H[k]} \approx \E C_N(t)^2.
\]
So the order of $D_N(t)$ is always dominated by that of $C_N(t)$. So we only need to focus on $A_N(t)$, $B_N(t)$ and $C_N(t)$. Then  the rest of the proof can be carried out  using the following consequence of \citet{bai:taqqu:2013:1-multivariate}:
\begin{equation}\label{eq:NCLT joint}
\left(\frac{1}{N^{(H-1)m+1}}\sum_{n=1}^{[Nt]}     H_m(Y(n)), ~m=1,\ldots,k \right) \ConvFDD  \Big(c_m Z_{H, m}(t), ~ m=1,\ldots,k \Big)
\end{equation}
where the Hermite processes $Z_{H,m}(t)$'s are defined through the same Brownian motion in (\ref{eq:Herm process}).

\end{proof}

\begin{proof}[Proof of Theorem \ref{Thm:near higher order non-center}]
The proof is done  as the proof of Theorem \ref{Thm:near higher order} based on (\ref{eq:CLT joint}) and (\ref{eq:NCLT joint}), and  we thus only provide an outline.
By Taylor expansion of $G_\infty$ in (\ref{eq:ana expan G_infty}), and since   $\E G(Y(n))=g_0=G_\infty(0)=0$ and   the Hermite rank is $k$, we have
\[
\E G(Y(n)+x_N)=G_\infty(x_N)=\frac{g_k}{k!}  x_N^k+O(x_N^{k+1}).
\]
Adding the leading term above to (\ref{eq:S_{[Nt]}}), we have
\begin{align}\label{eq:sum no center}
\widetilde{S}_{[Nt]}(x_N)=Z_N(t)+A_N(t)+B_N(t)+C_N(t)+D_N(t) +H.O.T.=Z_N(t)+A_N(t)+B_N(t)+R_N(t) +H.O.T.
\end{align}
where
\[
Z_N(t)=\frac{g_k }{k!} (x_N^k [Nt]) +H.O.T.
\]

In the central limit case where $H<1-1/(2k)$,   the term $B_N(t)$  is negligible compared to $A_N(t)$ and $R_N(t)$ as in the proof of Theorem \ref{Thm:near higher order} above.
We thus compare the orders $x_N^k N$ of $Z_N(t)$, $x_N^{k-1} N^H$ of $A_N(t)$ and $N^{1/2}$ of $R_N(t)$. Indeed, if $x_N\ll N^{-1/(2k)}$, then $x_N^{k}N\ll N^{1/2}$ and $x_N^{k-1} N^H \ll N^{1/2}$, so only $R_N(t)$ contributes; if $x_N\approx N^{-1/(2k)}$, then $x_N^{k}N\approx N^{1/2}$ and $x_N^{k}N\ll N^{1/2}$, so  $Z_N(t)$ and  $R_N(t)$ contribute; if  $x_N\gg N^{-1/(2k)}$, then $x_N^{k}N\gg N^{1/2}$ and $x_N^{k}N\gg x_N^{k-1} N^H$, so only $Z_N(t)$ contributes.

A similar analysis can be carried out in the non-central limit case. As in the proof of Theorem \ref{Thm:near higher order},  the term $B_N(t)$ only contributes in the case $x_N\approx N^{H-1}$.  The asymptotic regimes in the other cases are determined by   comparing the orders $x_N^k N$ of $Z_N(t)$, $x_N^{k-1} N^H$ of $A_N(t)$ and $N^{(H-1)k+1}$ of $C_N(t)$.
\end{proof}

\begin{proof}[Proof of Theorem \ref{Thm:near higher order centering}]
The proof is again similar to those of Theorem \ref{Thm:near higher order} and \ref{Thm:near higher order non-center}, and we thus only provide an outline.
Since $G(\cdot)$ is a polynomial, in (\ref{eq:A-D}) the $\infty$ in the sum defining $D_N(t)$ is replaced by a finite integer. This is important since the arguments leading to (\ref{eq:approx clt}) can no longer be applied.

The key is to note that $-\bar{Y}_N$ is associated with the order $N^{H-1}$, which replaces $x_N$ in all the terms in (\ref{eq:sum no center}). More precisely,
in the central limit case where $H<1-1/(2k)$,  comparing the order $N^{k(H-1)}$ of $Z_N(t)$,  the order $N^{(k-m)(H-1)+H[m]}$ of $A_N(t)$ and $B_N(t)$, $0<m<k$,  and the order $N^{1/2}$ of $R_N(t)$, where $H[m]$ is as in (\ref{eq:H(m)}), one can find that $R_N(t)$ is the contributing term. The conclusion then follows from (\ref{eq:CLT joint})  and the Slutsky Lemma using the continuity of the coefficients $c_m(\cdot)$. In the non-central limit case where $H>1-1/(2k)$, we compare the order $N^{k(H-1)}$ of $Z_N(t)$,  the order $N^{k (H-1)+1}$ of $A_N(t),B_N(t),C_N(t)$, $0<m\le k$,       and the order $N^{H[k+1]}$ of $D_N(t)$. In this case, the terms $A_N(t)$ and $B_N(t)$ contribute.

\end{proof}

\bigskip
\textbf{Acknowledgment.} This work was partially supported by the NSF grant  DMS-1309009  at Boston University.

\bibliographystyle{plainnat}
\bibliography{Bib}
\bigskip\noindent
Shuyang Bai\\
Department of Statistics\\
University of Georgia\\
Athens, GA 30606\\
\emph{bsy9142@uga.edu}

\medskip\noindent
Murad S. Taqqu\\
Department of Mathematics and Statistics\\
Boston University\\
Boston, MA 02215\\
\emph{murad@bu.edu}

\end{document}